\documentclass[11pt]{amsart}
\usepackage{amssymb}
\usepackage{amsmath}
\usepackage{tikz-cd}

\newtheorem{theorem}{Theorem}

\newtheorem{lemma}[theorem]{Lemma}
\newtheorem{prop}[theorem]{Proposition}

\newtheorem{claim}[theorem]{Claim}

\theoremstyle{definition}
\newtheorem{definition}[theorem]{Definition}

\newcommand{\C}{\mathcal{C}}
\newcommand{\V}{\mathcal{V}}
\newcommand{\HF}{\textnormal{HF}}

\newcommand{\res}{\upharpoonright}
\newcommand{\inv}{^{-1}}

\title{Computable vs Descriptive Combinatorics of Local Problems on Trees}
\author{Felix Weilacher}
\email{}

\begin{document}

\maketitle

\begin{abstract}
    We study the position of the computable setting in the ``common theory of locality'' developed in \cite{G+_trees} and \cite{BCG_Trees2} for local problems on $\Delta$-regular trees, $\Delta \in \omega$. We show that such a problem admits a computable solution on every highly computable $\Delta-$regular forest if and only if it admits a Baire measurable solution on every Borel $\Delta$ -regular forest. We also show that if such a problem admits a computable solution on every computable maximum degree $\Delta$ forest then it admits a continuous solution on every maximum degree $\Delta$ Borel graph with appropriate topological hypotheses, though the converse does not hold.
\end{abstract}

\markboth{FELIX WEILACHER}{COMP. VS DESCR. COMB. OF LOCAL PROBLEMS ON TREES}

\section{Introduction}

In this paper, we consider locally checkable labeling problems (LCLs) on regular trees. We will use the following formulation of such LCLs from \cite{G+_trees}. Colorings in this formulation are on so-called ``half edges'', and we will need a nonstandard formalization of the notion of a graph to fully accommodate these. Fix $\Delta \geq 2$.

\begin{definition}\label{def:graph}
    A \textit{$\Delta$-regular graph} $G$ with vertex set $X$ and edge set $E$ is a subset $G \subset X \times E$ such that for each $x \in X$, there are exactly $\Delta$ $e \in E$ with $(x,e) \in G$, for each $e \in E$, the number of $x \in X$ with $(x,e) \in G$ is either 1 or 2, and for each $x \neq y$ in $X$, there is at most 1 $e \in E$ with $(x,e),(y,e) \in G$.
    
    An element $(x,e)$ $G$ is called a \textit{half edge incident to} $x$, $x$ an \textit{endpoint} of $e$, and $e$ \textit{incident} to $x$. If a given edge $e$ has two endpoints, say $x,y \in X$, we call it a \textit{true} edge, and say $x$ and $y$ are \textit{adjacent} and that $y$ is a \textit{neighbor} of $x$. In this case we may also identify the half edge $(x,e)$ with the ordered pair $(x,y)$. We will also call the half edges $(x,e)$ and $(y,e)$ the half edges \textit{comprising} $e$. Otherwise we call $e$ a \textit{virtual} edge and $(x,e)$ a \textit{virtual half edge}.
    
    If $x$ is a vertex of $G$, the \textit{degree} of $x$, denoted $\deg(x)$, is the number of neighbors of $x$. Call $G$ \textit{truly} $\Delta$-regular if the degree of each vertex is $\Delta$. This is equivalent to $G$ having no virtual edges.
\end{definition}

We use the definition of adjacency for vertices above to port standard graph theoretic terminology into this model. E.g. path, path distance, cycle, connected component. We call a graph $G$ \textit{acyclic} or a \textit{forest} if it has no cycles. We call a set of vertices $G$-\text{invariant} if it is a union of $G$-connected components. We use $B(-,r)$ to denote the radius $r$ (w.r.t. path distance) neighborhood of a vertex or set of vertices.

Observe also that if our $G$ is truly $\Delta$-regular, the adjacency relation gives a $\Delta$-regular graph on $X$ in the traditional sense, call it $\overline{G}$. In this case the identification in paragraph 2 of the above definition actually gives a bijection between $G$ and $\overline{G}$ where half edges in $G$ go to ordered edges in $\overline{G}$. We can thus identify $G$ and $\overline{G}$ in this case.

We can now describe our LCLs:

\begin{definition}\label{def:lcl}
	An \textit{LCL on $\Delta$-regular graphs} $\Pi$ is a triple $(\Sigma,\mathcal{V},\mathcal{E})$, where $\Sigma$ is a finite set, and $\mathcal{V}$ and $\mathcal{E}$ are sets of size $\Delta$ and $2$, respectively, multisets of labels from $\Sigma$. 
	
	Let $G$ be a $\Delta$-regular graph and $c:G \rightarrow \Sigma$. (That is, $c$ is a labeling of the half edges of $G$.)
	    \begin{itemize}
	        \item If $e$ is a true edge of $G$, write $c(e)$ to denote the size 2 multiset $\{c(x,e),c(y,e)\}$, where $x$ and $y$ are the endpoints of $e$.
	        \item If $x$ is a vertex of $G$, write $c(x)$ to denote the size $\Delta$ multiset of colors given to the half edges meeting $x$.
	    \end{itemize}
    We say $c$ is a \textit{$\Pi$-coloring} of $G$ if for each true edge $e$ of $G$, $c(e) \in \mathcal{E}$, and for each vertex $x$ of $G$, $c(x) \in \mathcal{V}$.
\end{definition}

Much exciting recent work has found connections between the study of LCLs in various settings. For example, in the LOCAL model \cite{linial_LOCAL}, one imagines vertices as processors and edges as communication channels, where each processor must decide on labels for its incident half edges based on a small amount of communication with its neighbors so that the result of all these decisions amounts to a $\Pi$-coloring for a given $\Pi$. In the descriptive setting, the vertex and edge sets are Polish spaces, $G$ is a descriptive-set-theoretically ``nice'' relation on them (e.g. Borel), and one looks for $\Pi$-colorings which are similarly ``nice''. The systematic study of connections between these two settings was introduced in \cite{Bernshteyn.distributed}. A starting point for this paper is the work in \cite{G+_trees} and \cite{BCG_Trees2}, which carried out a much more complete study of connections between these settings and others in the special case where $G$ is acyclic. 

Less explored are connections between these fields and the computable setting, where one is given a graph which is in some sense computable, and looks for similarly computable $\Pi$-colorings. Work of Qian and Weilacher \cite{QW_ASI} established some parallels between certain parts of this setting and of the descriptive setting, but failed to find any direct relationship. The main results of this paper show that such direct relationships do exist if one again restricts to the special case of acyclic graphs.

We now formally describe some of the settings which were described informally above. In (3) below and throughout the paper, $\HF$ denotes the set of hereditarily finite sets, which we use as our domain for discussing computability. Note that the definitions below make sense for general classes of graphs, but in this paper we are chiefly concerned with complexity classes for acyclic graphs.

\begin{definition}
	Let $\Pi = (\Sigma,\mathcal{V},\mathcal{E})$ be an LCL on $\Delta$-regular graphs.
	\begin{enumerate}
		\item We call a truly $\Delta$-regular graph $G$ \textit{Borel} if its vertex set $X$ is a Polish space and the associated $\overline{G}$ (see the comment after Definition \ref{def:graph}) is Borel in $X \times X$. We call a coloring $c:G \rightarrow \Sigma$ \textit{Borel} if it is Borel as function on $\overline{G} \subset X \times X$. We say $\Pi$ is in the class BAIRE if every acyclic such $G$ admits a Borel $\Pi$-coloring off a $G$-invariant meager Borel subset of vertices. I.e, there is a $G$-invariant comeager Borel set $C \subset X$ such that the graph identified with $\overline{G} \cap (C \times C)$ admits a Borel $\Pi$-coloring.
		
		\item (\cite{bernshteyn_continuous},\cite{BCG_Trees2}) We call a Borel $\Delta$-regular graph $G$ as in (1) \textit{continuous} if $X$ is zero-dimensional and for any clopen $U \subset X$, the set of neighbors of elements of $U$ is clopen. We call a coloring $c: G \rightarrow \Sigma$ \textit{continuous} if it continuous as a function on $\overline{G} \subset X \times X$ with the subspace topology. We say $\Pi$ is in the class CONTINUOUS if every acyclic such $G$ admits a continuous $\Pi$-coloring.

		\item We call a $\Delta$-regular graph $G$ \textit{computable} if its vertex and edge sets $X$ and $E$ are computable subsets of $\HF$ and $G$ is a computable subset of $X \times E$. We call a coloring $c : G \rightarrow \Sigma$ computable if it is computable as a function (we may of course assume $\Sigma \subset \HF$.) We say $\Pi$ is in the class COMPUTABLE if every computable acyclic $\Delta$-regular graph admits a computable $\Pi$-coloring.

        \item We call a computable $\Delta$-regular graph $G$ as in (3) \textit{highly computable} if the degree function $\deg:X \rightarrow \omega$ is computable. Since there are $\Delta$ edges incident to each $x \in X$ and $\deg(x)$ is the number of true such edges, this is equivalent to the set of true edges being computable. In particular, truly $\Delta$-regular computable graphs are highly computable. We say $\Pi$ in in the class HCOMP if every highly computable acyclic $\Delta$-regular graph admits a computable $\Pi$-coloring.
    \end{enumerate}
\end{definition}

We can now state our results. In \cite{QW_ASI}, Qian and Weilacher ask whether, for a general nice class of graphs, LCLs solvable in the highly computable setting are always solvable Baire-measurably. Our first result confirms this for regular forests.

\begin{theorem}\label{th:main}
    $\textnormal{BAIRE} = \textnormal{HCOMP}$.
\end{theorem}

Our second result concerns the class COMPUTABLE. This class does not strongly parallel any in the descriptive or LOCAL settings; the inability in computable but not highly computable graphs to effectively quantify over the neighbors of a vertex, a process which is trivial in these other settings, severely restricts the problems one can make headway on. Still, the class does contain nontrivial problems such as $(\Delta+1)$-proper vertex coloring \cite{schmerl_brooks}, making it meaningful to establish the following one-sided implication:

\begin{theorem}\label{th:maincomp}
    $\textnormal{COMPUTABLE} \subset \textnormal{CONTINUOUS}$, but the reverse inclusion does not hold.  Moreover, if an LCL on $\Delta$-regular graphs is in \textnormal{COMPUTABLE}, it can be solved continuously on any $\Delta$-regular continuous graph, and computably on any $\Delta$-regular computable graph.
\end{theorem}

It should be noted that, by the equivalences shown in \cite{G+_trees} and \cite{BCG_Trees2} between our descriptive set theoretic classes of interest and various classes from other settings (such as the LOCAL model), Theorems \ref{th:main} and \ref{th:maincomp} also provide connections between these other settings and computable combinatorics. For example, an LCL is in HCOMP if and only if it can be solved by a deterministic LOCAL algorithm on trees in time $O(\log(n))$, and if an LCL is in COMPUTABLE, it can be solved by a deterministic LOCAL algorithm on trees in time $O(\log^*(n))$.

\section{Highly Computable Forests}

In this section we prove Theorem \ref{th:main}. Key to the equality is the following combinatorial condition:

\begin{definition}\label{def:full}
    \begin{itemize}
        \item Let $l \in \omega$. A \textit{$\Delta$-regular path of length $l$} is a $\Delta$-regular tree with $l+1$ vertices, say $x_0,\ldots,x_l$, with $x_i$ and $x_{i+1}$ adjacent for each $i < l$, and all other edges being virtual. $x_0$ and $x_l$ are called the \textit{endpoints} of the path.
        \item Let $\Pi = (\Sigma,\mathcal{V},\mathcal{E})$ be an LCL on $\Delta$-regular graphs.  Let $l \geq 1$. Let $\mathcal{V}' \subset \mathcal{V}$. We say $\mathcal{V}'$ is \textit{$l$-full} if whenever $P$ is a $\Delta$-regular path of length at least $l$ and the half edges incident to its endpoints, say $x$ and $y$, are precolored so that $c(x),c(y) \in \mathcal{V}'$, this can be extended to a $\Pi$-coloring of $P$ so that $c(z) \in \mathcal{V}'$ for every vertex $z$ of $P$. We say $\Pi$ is \textit{$l$-full} if it admits some nonempty such $\mathcal{V}'$. We say $\Pi$ is \textit{full} if it is $l$-full for some $l$.
    \end{itemize}
\end{definition}

Recall the second bullet point in Definition \ref{def:lcl} for the meaning of $c(x)$ when $x$ is a vertex. The fullness condition is due to Anton Bernshteyn, who also has shown \cite{G+_trees} that an LCL on $\Delta$-regular graphs is in BAIRE if and only if it is full. It thus suffices for us to prove:

\begin{theorem}\label{thm:hcifffull}
    An LCL on $\Delta$-regular graphs is in HCOMP if and only if it is full.
\end{theorem}

The proof in \cite{G+_trees} of the reverse direction of Bernshteyn's result uses a construction called a \textit{toast}:

\begin{definition}\label{def:toast}
    Let $G$ be a $\Delta$-regular graph and $l \in \omega$. An \textit{$l$-toast} for $G$ is a collection $\C$ of nonempty finite $G$-connected sets of vertices (called \textit{pieces}) such that
    \begin{enumerate}
        \item For distinct $C,D \in \C$, either $B(C,l) \subset D$, $B(D,l) \subset C$, or the path distance from $C$ to $D$ is greater than $l$.
        \item If vertices $x,y$ are in the same $G$-component, there is some $C \in \C$ containing both of them.
    \end{enumerate}
\end{definition}

If $G$ is a computable graph, it makes sense to ask that a toast for $G$ be computable. In \cite{G+_trees}, Proposition 6.4, a straightforward algorithm is presented which uses a $(2l+2)$-toast for a $\Delta$-regular forest, say $T$, to produce a $\Pi$-coloring of $T$ given that $\Pi$ is $l$-full. One can easily see that this will produce a computable $\Pi$-coloring if the input toast is computable, so it suffices for our reverse direction to prove the following, whose origins seem to be in \cite{bean} (Theorem 3).

\begin{theorem}\label{thm:comptoast}
    Let $G$ be a highly computable $\Delta$-regular graph. $G$ admits a computable $l$-toast for every $l \in \omega$.
\end{theorem}

\begin{proof}
Fix $l$. WLOG the vertex set of $G$, say $X$, is a subset of $\omega$. We will construct our toast recursively in $\omega$-stages. At each stage, we will have only finitely many pieces in our toast.

At stage $n$, if $n \not\in X$, we do nothing. If $n \in X$, we check our finitely many pieces to see if any contains $B(n,1)$. If so, we do nothing.

If not, we find the least $r > 0$ such that $B(n,r)$ can be added to our toast without violating (1). This exists since our toast so far is finite (so, in fact, there is a tail of $r$'s which will work) and it can be determined effectively since $G$ is highly computable, which means we can computably output the finite graph $G \res B(n,r)$ given $n$ and $r$. We then add $B(n,r)$ to our toast. 

Call the resulting set of finite connected sets of vertices $\C$. $\C$ is computable since at stage $n$, if we add a new set to $\C$, that set contains the vertex $n$. Thus, to test whether a given finite set $C$ is in $\C$, we only need to run our algorithm until stage $\max(C)$. (1) is satisfied since we made sure it was satisfied when adding each new piece. For each $n \in X$, at stage $n$ we made sure there was a piece of toast containing $B(n,1)$. This implies (2) by induction on the path distance between $x$ and $y$: If $z$ is a neighbor of $y$ and we have a piece of toast $C$ containing $x$ and $z$ and a piece of toast $D$ containing $B(y,1)$, then $C \cap D \neq \emptyset$, so one must contain the other, and then the larger one contains $x$ and $y$.
\end{proof}

We now prove by contrapositive the forward direction of Theorem \ref{thm:hcifffull}. The following proof is inspired by \cite{MR_matching}, which used a similar argument to construct, for a given $k \geq 2$, a computable truly $k$-regular computably bipartite graph with no computable perfect matching:

\begin{prop}\label{prop:hcconstruction}
        If an LCL on $\Delta$-regular graphs $\Pi = (\Sigma,\mathcal{V},\mathcal{E})$ is not full, there is a computable truly $\Delta$-regular forest (thus highly computable) with no computable $\Pi$-coloring.
\end{prop}

\begin{proof}

If $\Pi$ is not full, then for every nonempty $\mathcal{V}' \subset \mathcal{V}$ and $l' \geq 1$, there are $l \geq l'$, multisets, say $a,b \in \mathcal{V}'$, and labels, say $\alpha, \beta \in \Sigma$ appearing in $a$ and $b$ respectively such that the following holds: Let $P$ be a $\Delta$-regular path of length $l$, say with vertices $x_0,\ldots,x_l$ in order. Label the half edges incident to $x_0$ such that $c(x_0) = a$ and $c(x_0,x_1) = \alpha$, and likewise for $x_l, b, x_{l-1}$ and $\beta$ respectively. Then if this labeling is extended to a $\Pi$-coloring $c$ of $P$, some vertex $x_i$ has $c(x_i) \not\in \mathcal{V}'$. Note that if $\mathcal{V}' = \mathcal{V}$, this means there is no such extension. 

Since $\Sigma$ is finite, we can say instead that for all nonempty $\mathcal{V}' \subset \mathcal{V}$, there exists $a,b \in \mathcal{V}'$ and $\alpha,\beta \in \Sigma$ which have this property for infinitely many $l$. Fix choices of such $a,b,\alpha$, and $\beta$ for each $\mathcal{V}'$, and call them \textit{bad} for $\mathcal{V}'$. Also call the infinite set of $l$ witnessing this badness the set of \textit{bad} $l$ for $\mathcal{V}'$. Note that determining whether a given $l$ is bad for a given $\mathcal{V}'$ is computable: It can be done by checking all possible colorings of a $\Delta$-regular path of length $l$ with labels from $\Sigma$, of which there are only finitely many.


Fix $X,E \subset \HF$ disjoint infinite computable sets, and a computable well order of $\HF$ of type $\omega$. $X$ and $E$ will be our vertex and edge sets respectively. Fix an effective enumeration $\{\phi_t \mid t \in \omega \}$ of the set of partial recursive functions $X \times E \rightarrow \Sigma$. (These are candidates for computable colorings of the half edges of our graph.) Fix also an effective enumeration $\{t_n \mid n \in \omega \}$ of $\omega$ which lists each natural number infinitely many times. We are about to describe a recursive construction of $\Delta$-regular forests $\emptyset = T_0 \subset T_1 \subset \cdots$ with vertex sets contained in $X$. The vertex set for each $T_n$ will be finite. It will happen that in our construction every $x \in X$ will eventually become a vertex and be made to have degree $\Delta$. The limit $T = \bigcup_n T_n$ will then be a computable truly $\Delta$-regular forest. Our construction will be designed so that $T$ does not admit a computable $\Pi$-coloring.

Alongside $T$, we will recursively build a $T$-invariant total computable function $C:X \rightarrow \omega$. The idea is that the vertices in $C\inv(t)$ will be responsible for ensuring that $\phi_t$ is not a $\Pi$-coloring of $T$. At stage $n$, the vertices on which $C$ has been defined will be exactly those in the vertex set of $T_n$. 

At a typical stage in the construction, we will take currently unused (not in the vertex set of our current finite tree) vertices, define $C$ on them, and give them $\Delta$ incident half edges. We will refer to these as ``new vertices'', and always choose the least (according to our previously fixed order of $\HF$) available. For example, if $x$ has degree 0 in $T_n$, and in the definition of $T_{n+1}$ we say ``add $\Delta$ new vertices as neighbors to $x$'', it means: let $y_0,\ldots,y_{\Delta-1}$ be the least elements of $X$ not in the vertex set of $T_n$, let $e_0,\ldots,e_{\Delta-1}$ be the virtual edges incident to $x$ in $T_n$, and add the half edges $(y_i,e_i)$ to $T_{n+1}$ for $i < \Delta-1$. The exact same treatment will be given to edges. For example, in the previous example, we would probably want to finish by saying ``add $\Delta-1$ new virtual edges incident to each $y_i$'' to maintain $\Delta$-regularity. This would mean: let $f_{i,j}$ for $i < \Delta$, $j < \Delta - 1$ be the least elements of $E$ not in the edge set of $T_n$, and add the half edges $(y_i,f_{i,j})$ to $T_{n+1}$ for each $i,j$. 

In both steps of the previous example, our language leaves some ambiguity about exactly which new vertices/edges are connected to which old edges/vertices. The exact decision here will never matter (except of course in that it should be made in a computable way).

If the stage of our construction is $n$, $C$ will always be defined to be $t_n$ on new vertices. Thus, our construction will essentially build $\omega$ forests in parallel, one for each $t$. Let us fix $t$, and describe the $\omega$ steps in the construction of the $t$-th forest, i.e, $T \res C\inv(t)$.

\begin{figure}
    \centering
    \includegraphics[width = 0.5\textwidth]{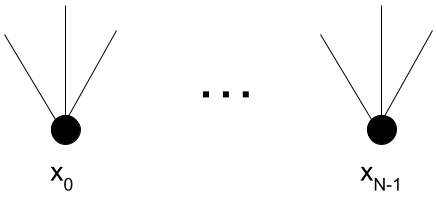}
    \caption{Stage 0 of the construction for a fixed $t$ for $\Delta = 3$.}
    \label{fig:0}
\end{figure}

Fix a very large $N_0 >> |\Sigma|$ to be determined later, independent of $t$. In the 0-th stage of the construction, place $N_0$ new vertices, say $x_0,\ldots,x_{N_0-1}$, in $C\inv(t)$. Also give each $x_i$ $\Delta$ new virtual incident edges. This is shown in Figure \ref{fig:0}.

We also initialize several variables: Set $\mathcal{V}' = \emptyset$ and $N = N_0$. Also, for each $i < N$, let $P_i$ be the $\Delta$-regular path of length 0 whose unique vertex is $x_i$.

At the start of an arbitrary successor stage, say $m+1$, suppose we have some $\mathcal{V}' \subset \Sigma$ and $N \in \omega$ very large. Suppose also that we have $\Delta$-regular paths $P_i$ for $i < N$ and vertices $y_i^c$ for $i < N$ and $c \in \mathcal{V}'$ such that
\begin{enumerate}
    \item The $P_i$'s and $y_i^c$'s all lie in distinct $C\inv(t)$-components.
    \item For each $i$, if $\phi_t$ restricted to the half edges meeting $P_i$ is a $\Pi$-coloring of $P_i$, there must be some vertex $x$ of $P_i$ with $\phi_t(x) \not\in \mathcal{V}'$.
    \item $\phi_t(y_i^c) = c$ for each $i$ and $c$.
\end{enumerate}
Note that this is all satisfied after stage 0. 

\begin{figure}
    \centering
    \includegraphics[width = 0.5\textwidth]{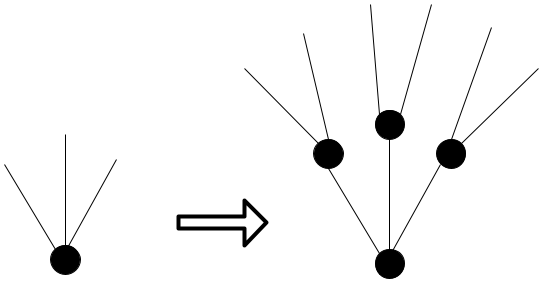}
    \caption{The ``uninteresting case'' for a step of the construction for a fixed $t$, for $\Delta = 3$.}
    \label{fig:default}
\end{figure}

We start our successor stage by running the Turing machine computing $\phi_t$ for $m$ steps on all of the half edges meeting all of the $P_i$'s. 

If $\phi_t$ fails to converge in $\leq m$ steps on one of these inputs, or if it converges on all of them but the result fails to be a $\Pi$-coloring of some $P_i$, we do what is shown in Figure \ref{fig:default}: For each virtual edge $e$ with unique endpoint in $C\inv(t)$, we make $e$ true by adding a new vertex, say $x$, as its other endpoint, then to maintain $\Delta$-regularity we add $\Delta-1$ new edges as virtual edges incident to this $x$. We call this the ``uninteresting case''.

The ``interesting case'' is of course if $\phi_t$ does converge on all these inputs in $\leq m$ steps, and the result is a $\Pi$-coloring for each $P_i$. In this case, by condition (2) above, there is a vertex in each $P_i$, say $z_i$, with $\phi_t(z_i) \not\in \mathcal{V}'$. Let $d \in \mathcal{V} \setminus \mathcal{V}'$ such that $\phi_t(z_i) = d$ for at least $N/|\mathcal{V}|$ $i$'s. Update $\mathcal{V}'$ to $\mathcal{V}' \cup \{d\}$.

Now $\mathcal{V}'$ is nonempty, so let $a,b \in \mathcal{V}'$ with $\alpha \in a$ and $\beta \in b$ be bad for $\mathcal{V}'$. Between the $z_i$'s and the $y_i^c$'s, and using condition (1) above, we can choose degree $\Delta$ vertices $v_i$ and $w_i$ for $i < M := N/(2|\mathcal{V}|)$ such that $\phi_t(v_i) = a$ for each $i$, $\phi_t(w_i) = b$ for each $i$, and all $2M$ of these vertices lie in distinct $C\inv(t)$-components. (The factor of $1/2$ is needed for the case $a = b$.)

\begin{figure}
    \centering
    \includegraphics[width = 0.7\textwidth]{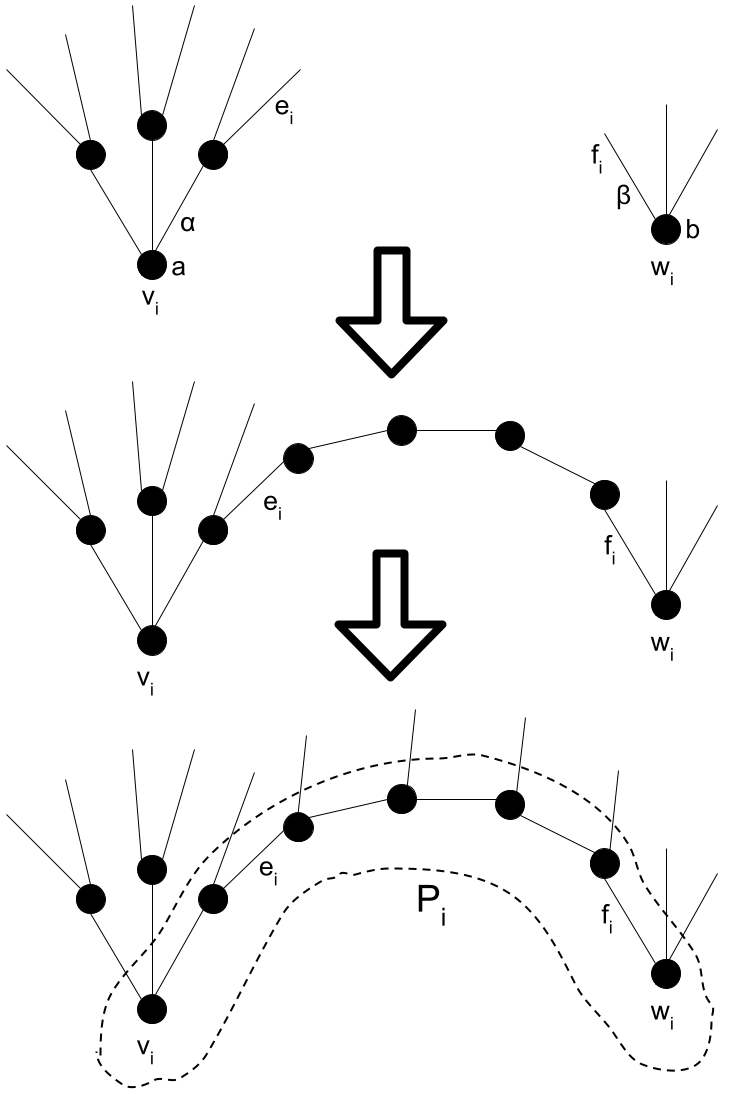}
    \caption{The ``interesting case'' for a step of the construction for a fixed $t$. The top and bottom of the image represent the before and after states respectively. The dashed line encloses one of the new paths, $P_i$.}
    \label{fig:join}
\end{figure}

Now for each $i < M / 2$, pick $e_i$ and $f_i$ virtual edges meeting the components of $v_i$ and $w_i$ respectively such that if $e$ is the first edge in the path from $v_i$ to $e_i$, then $\phi_t(v_i,e) = \alpha$, and likewise for $w_i$, $f_i$, and $\beta$. Note that these exist since our components are finite $\Delta$-regular trees. As shown in in the first step in Figure \ref{fig:join}, join $e_i$ and $f_i$ with a path of new vertices so that the resulting path between $v_i$ and $w_i$ has a total length which is bad for $\mathcal{V}'$. (Recall that this is possible as there are arbitrarily long bad lengths, and that determining if a length is bad is computable.) Call this new path $P_i$. Also, as shown in the second step in Figure \ref{fig:join}, add $\Delta-2$ new virtual incident edges to the newly added vertices in $P_i$ to maintain $\Delta$-regularity. Condition (2) is satisfied for each $P_i$ by definition of bad.

We can now finish this stage of the construction with the updated value of $N$ being $M/2$. We still need to define the new $y_i^c$'s. Use the unused $v_i$'s and $w_i$'s for $c \in \{a,b\}$. For $c \not\in \{a,b\}$, use the first $M/2$ $z_i$'s with $\phi_t(z_i) = d$ if $c = d$, and the first $M/2$ $y_i^{c}$'s from the previous stage of the construction if $c \neq d$ (i.e, if $c$ was in the previous $\mathcal{V}'$.) (1) is clear by construction.

Thus the description of the construction is complete. We can also see now that the correct value of $N_0$ to take was $N_0 > (4|\mathcal{V}|)^{|\mathcal{V}|}$, since the value of $N$ was divided by $4|\mathcal{V}|$ each time the interesting case occurred, and it cannot occur more than $|\mathcal{V}|$ times for a fixed $t$ (Since $|\mathcal{V}'|$ increases by one each time it does).

\begin{claim}\label{claim:computable}
$T$ is a computable truly $\Delta$-regular forest.
\end{claim}
\begin{proof}
At each stage in the construction, we either added new trees as components, added new degree 1 neighbors to vertices, or joined pairs of components along a single path. Thus the components at each stage were trees, and so the limit $T$ must be acyclic.

For each fixed $t$, the interesting case could only have occurred finitely many times as noted before the claim. Thus each vertex, once added to $C\inv(t)$, must have eventually experienced the uninteresting case, at which point it would be made degree $\Delta$ by construction. Since we added new vertices at the first stage for each $t$, and we always take these to be the least available, each $x \in X$ is eventually added to some $C\inv(t)$, and subsequently made to have degree $\Delta$. Thus $T$ is truly $\Delta$-regular and its vertex set is $X$, which was chosen to be computable. Similarly the edge set of of $T$ is $E$, also chosen to be computable.

Finally, $T$ is clearly computable: The construction of the $T_n$'s was clearly recursive, so to check whether a given $(x,e) \in X \times E$ is in $T$, we can run the construction until $x$ is added to the vertex set, at which point it gets $\Delta$ incident edges, and we can check if any of these is $e$.
\end{proof}

\begin{claim}
$T$ has no computable $\Pi$-coloring.
\end{claim}

\begin{proof}
Suppose $\phi_t$ is a $\Pi$-coloring of $T$, and consider our construction for this fixed value of $t$. In between each interesting case stage of our construction, we have finitely many $\Delta$-regular paths $P_i$, and are waiting for $\phi_t$ to converge to a $\Pi$-coloring of them. Since $\phi_t$ is in fact a $\Pi$-coloring of $T$, we are always guaranteed that this will eventually happen. That is, the interesting case will occur unboundedly many times for this $t$, but as noted before Claim \ref{claim:computable}, it can only occur $|\mathcal{V}| < \omega$ times.
\end{proof}

Thus $T$ is as desired.

\end{proof}

\section{Computable Forests}

In this section we prove Theorem \ref{th:maincomp}. As in the previous section, the key is to find a combinatorial condition characterizing our class.

\begin{definition}\label{def:greedy}
    Let $\Pi = (\Sigma,\mathcal{V},\mathcal{E})$ be an LCL on $\Delta$-regular graphs. Let $\Sigma' \subset \Sigma$. We say $\Sigma'$ is \textit{greedy} if the following holds: Let $k \in \{0,\ldots,\Delta\}$, and consider the $\Delta$-regular tree $H$ with vertices $x$ and $y_i$ for $i < k$ with $x$ and $y_i$ adjacent for each $i$ and all other edges being virtual. Precolor each half edge $(y_i,x)$ with a label from $\Sigma'$. Then this can be extended to a $\Pi$-coloring of $H$ such that $c(x,e) \in \Sigma'$ for each virtual edge $e$ incident to $x$. (See Figure \ref{fig:greedy}.) We say $\Pi$ is \textit{greedy} if it admits some greedy $\Sigma'$.
\end{definition}

\begin{figure}
    \centering
    \includegraphics[width = 0.9\textwidth]{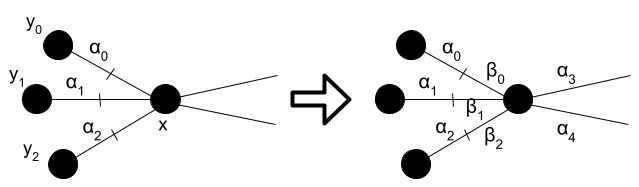}
    \caption{An illustration of Definition \ref{def:greedy} with $\Delta = 5$ and $k = 3$. True edges are divided with slashes to show the two half edges comprising them. The left hand figure is given, and the right hand figure is the needed extension. The $\alpha$'s are colors from $\Sigma'$, whereas the $\beta$'s can be any elements of $\Sigma$. Definition \ref{def:greedy} also technically asks for colors for the other half edges incident to the $y_i$'s, but these do not turn out to be relevant and so are not drawn.}
    \label{fig:greedy}
\end{figure}

In analogy with Theorem \ref{thm:hcifffull}, we prove:

\begin{theorem}\label{thm:ciffgreedy}
    An LCL on $\Delta$-regular graphs is in COMPUTABLE if and only if it is greedy.
\end{theorem}

We start with the reverse direction. Actually, to prove the final claim of Theorem \ref{th:maincomp}, we deal with the stronger statement where our graph need not be acyclic:

\begin{prop}\label{prop:greedytocomp}
        Let $\Pi = (\Sigma,\V,\mathcal{E})$ be an LCL on $\Delta$-regular graphs, and assume $\Pi$ is greedy. Then any computable $\Delta$-regular graph admits a computable $\Pi$-coloring.
\end{prop}

\begin{proof}
    Let $G$ be such a graph, say with vertex and edge sets $X$ and $E$ respectively. WLOG $X \subset \omega$. Let $\Sigma' \subset \Sigma$ be greedy. We will recursively define our coloring $c$ in $\omega$ stages. At stage $n \in \omega$, we will color exactly the half edges incident to $n$ if $n \in X$, and do nothing otherwise. We will maintain that if $e \in E$ is such that exactly one of the half edges comprising $e$ has been colored so far, then it gets a color from $\Sigma'$.
    
    We now describe stage $n$ if $n \in X$: Let $y_i$ for $i < k$ be the neighbors of $n$ less than $n$. Then by inductive hypothesis, $c(y_i,n) \in \Sigma'$ for each $i$, and the $\Delta-k$ other incident edges to $x$, say $e_j$ for $j < \Delta - k$, not having a $y_i$ as their other endpoint have not been touched. Thus, by the definition of greedy, we may color the half edges incident to $x$ so that $c(x,e_j) \in \Sigma'$ for each $j$, thus maintaining our inductive hypothesis. 
\end{proof}

As is apparent from this proof, if an LCL on $\Delta$-regular graphs is greedy, it can be solved using a ``greedy'' algorithm. Such algorithms our typically easy to carry out in other combinatorial settings, making the analogue of Proposition \ref{prop:greedytocomp} easy to prove in these settings. For example (giving another of the parts of Theorem \ref{th:maincomp}):

\begin{prop}\label{prop:greedytocts}
    Let $\Pi = (\Sigma,\V,\mathcal{E})$ be an LCL on $\Delta$-regular graphs, and assume $\Pi$ is greedy. Then any continuous $\Delta$-regular graph admits a computable $\Pi$-coloring.
\end{prop}

\begin{proof}
    Let $G$ by such a graph, say with vertex set $X$. By \cite{bernshteyn_continuous}, there is a continuous proper coloring $d:X \rightarrow \omega$. We can then repeat the construction of the previous proposition, except that at stage $n \in \omega$ we color the half edges incident to all vertices in $d\inv(n)$. This works because each $d\inv(n)$ is independent, and the result is continuous since $d$ and $G$ are. (See, for example, Theorem 3.2 in \cite{BCG_Trees2}.)
\end{proof}

We now turn to the forward direction of \ref{thm:ciffgreedy}. Again we prove the contrapositive:

\begin{prop}
    If an LCL on $\Delta$-regular graphs $\Pi = (\Sigma,\V,\mathcal{E})$ is not greedy, there is a computable $\Delta$-regular tree with no computable $\Pi$-coloring.
\end{prop}

\begin{proof}
    Our construction will closely follow that in the proof of Proposition \ref{prop:hcconstruction}. Fix $X,E$ and the $\phi_t$'s as before. Once again, we will be recursively constructing a sequence of $\Delta$-regular forests with finite vertex sets, but really building a sequence of $\omega$ forests in parallel as kept track of by a recursively defined function $C:X \rightarrow \omega$. We will use the term ``new'' for added vertices and edges in the same way. Call the forest we are building $T$.
    
    Since $\Pi$ is not greedy, for every $\Sigma' \subset \Sigma$, there is a $k \in \{0,\ldots,\Delta\}$ and a sequence $\alpha_0,\ldots,\alpha_{k-1}$ from $\Sigma'$ witnessing that $\Sigma'$ is not greedy. That is, such that in the situation of Definition \ref{def:greedy}, if this $k$ is used and each half edge $(y_i,x)$ is given the color $\alpha_i$, then if this is extended to a $\Pi$-coloring of $H$, there is some virtual half edge $(x,e)$ such that $c(x,e) \not\in \Sigma'$. Fix choices of such $k$ and $(\alpha_i)_{i < k}$ for each $\Sigma'$ and call them \textit{bad} for $\Sigma'$.
    
    Again, let us fix $t \in \omega$ and describe the $\omega$ steps in the construction of $T \res C\inv(t)$.
    
    The 0-th stage will be similar: Fix a very large $N_0 >> |\Sigma|$ to be determined later, independent of $t$. In the 0-th stage, place $N_0$ new vertices, say $w_0,\ldots,w_{N_0-1}$, in $C\inv(t)$, and give each $\Delta$ new virtual incident edges.
    
    We also initialize several variables: Set $\Sigma' = \emptyset$ and $N = N_0$. Also set $x_i = w_i$ for each $i < N$.
    
    At the start of an arbitrary successor stage, say $m+1$, suppose we have some $\Sigma' \subset \Sigma$,$N \in \omega$ very large, and vertices $x_i$ for $i < N$ and virtual half edges $(y_i^\alpha,e_i^\alpha)$ for $i < N$ and $\alpha \in \Sigma'$ such that:
    \begin{enumerate}
        \item The $x_i$'s and $y_i^\alpha$'s all lie in distinct $C\inv(t)$-components.
        \item For each $i$, if $\phi_t$ restricted to half edges meeting $B(x_i,1)$ gives a $\Pi$-coloring, there must be some virtual edge $e$ incident to $x$ with $\phi_t(x,e) \not\in \Sigma'$.
        \item For each $i$ and $\alpha$, $\phi_t(y_i^\alpha,e_i^\alpha) = \alpha$.
    \end{enumerate}
    
Note that this is all satisfied after stage 0.

We start our successor stage by running the Turing machine computing $\phi_t$ for $m$ steps on all the half edges meeting each $B(x_i,1)$. (Note that we can compute $B(x_i,1)$ since our vertex set so far is finite.)

If $\phi_t$ fails to converge in $\leq m$ steps on one of these inputs, or if it converges on all of them but the result fails to be a $\Pi$-coloring of some $B(x_i,1)$, we do nothing. We call this the ``uninteresting case''.

The ``interesting case'' is of course if $\phi_t$ does converge on all these inputs in $\leq m$ steps, and the result is a $\Pi$-coloring for each $B(x_i,1)$. In this case by condition (2) above, there is a virtual edge $f_i$ for each $i < N$ meeting $x_i$ with $\phi_t(x_i,f_i) \not\in \Sigma'$. Let $\beta \in \Sigma \setminus \Sigma'$ such that $\phi_t(x_i,f_i) = \beta$ for at least $N/|\Sigma|$ $i$'s. Update $\Sigma'$ to $\Sigma' \cup \{\beta\}$.

\begin{figure}
    \centering
    \includegraphics[width = 0.6\textwidth]{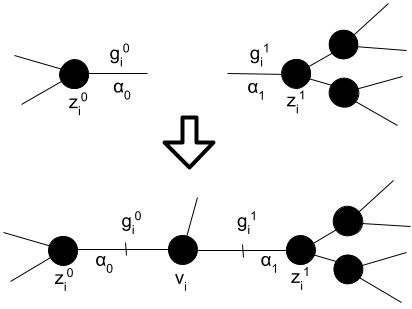}
    \caption{The ``interesting case'' for a step of the construction for a fixed $t$ with $\Delta = 3$. The top and bottom image represent the before and after states respectively.}
    \label{fig:gj}
\end{figure}

Let $k$ and $\alpha_0,\ldots,\alpha_{k-1} \in \Sigma'$ be bad for $\Sigma'$. Between the $(x_i,f_i)$'s and the $(y_i^\alpha,e_i^\alpha)$'s, and using condition (1) above, we can chose virtual half edges $(z_i^j,g_i^j)$ for $i < M := N / (\Delta |\Sigma|)$ such that the $z_i^j$'s all lie in distinct $C\inv(t)$-components and $\phi_t(z_i^j,g_i^j) = \alpha_j$ for each $j < k$. 

Now for each $i < M/2$, Let $v_i$ be a new vertex and add it as an endpoint to $g_i^j$ for each $j$, making all these edges true. Then, to maintain $\Delta$-regularity, add $\Delta-k$ new virtual incident edges to $v_i$. This is shown in Figure \ref{fig:gj}. 

We can now finish this stage of the construction with the updated value of $N$ being $M/2$. First set $x_i = v_i$ for each $i$. Condition(2) is then satisfied by definition of bad. We now define the $(y_i^\alpha,e_i^\alpha)$'s. If $\alpha = \alpha_j$ for some $j$, we can use the unused $(z_i^j,g_i^j)$'s. Else, if $\alpha = \beta$ we can use the first $M/2$ $(x_i,f_i)$'s, and if not, the first $M/2$ $(y_i^\alpha,e_i^\alpha)$'s from the previous stage of the construction.

Thus the description of the construction is complete. As in the proof of Proposition \ref{prop:hcconstruction}, it will be significant that the interesting case cannot occur more than $|\Sigma|$ times for a fixed $t$ (since $|\Sigma'|$ increases by one each time it does). As in that proof, the first application of this is that it tells us what to pick for $N_0$: We can take $N_0 > (2\Delta |\Sigma|)^{|\Sigma|}$, since $2\Delta|\Sigma|$ is by how much the value of $N$ was divided each time the interesting case occurred. 

\begin{claim}\label{claim:Tcomp2}
$T$ is a computable $\Delta$-regular forest.
\end{claim}

\begin{proof}
    At each stage in the construction, we either added new trees as components, did nothing, or joined separate components at a single new vertex. Thus $T$ is acyclic.
    
    $T$ has vertex and edge set $X$ and $E$ for the same reason as in the proof of Proposition \ref{prop:hcconstruction}. Also, since we maintain $\Delta$-regularity at each stage in the construction, to check if a pair $(x,e) \in X \times E$ is in $T$, we can run the construction until $x$ is added as a vertex, and then check if $e$ is one of its incident edges. Thus $T$ is computable.
\end{proof}

\begin{claim}
$T$ admits no computable $\Pi$-coloring.
\end{claim}

\begin{proof}
    Exactly as in the proof of Proposition \ref{prop:hcconstruction}, if some $\phi_t$ were a $\Pi$-coloring of $T$, we could conclude that the interesting case occurs infinitely often for this $t$, contradicting the statement before Claim \ref{claim:Tcomp2}.
\end{proof}

\end{proof}

We end this paper by showing that the inclusion in Theorem \ref{th:maincomp} is strict as promised. Actually by results from \cite{BCG_Trees2} and \cite{CKP_Speedup}, there is a sense in which being solvable by a greedy algorithm characterizes the class CONTNINUOUS as well: An LCL $\Pi$ is in CONTINUOUS if and only there is some $l \in \omega$ such that $\Pi$ can be solved on $\Delta$-regular forests by first finding a proper coloring of the distance $l$ graph (in which two vertices are adjacent if their distance in the original graph is at most $l$), then applying some constant-time local algorithm with that coloring as an input. It is easy to find proper colorings greedily, hence our statement. 

As in the proof of Proposition \ref{prop:greedytocts}, it is also easy to find proper colorings continuously. This is true even for the distance $l$ graph, hence the reverse direction in the result mentioned in the previous paragraph. However, in the computable setting, there is a big difference between $l = 1$ and $l > 1$; if a graph is computable but not highly computable, we cannot effectively determine which pairs of vertices have distance $l$ from eachother if $l > 1$. This difference will be key to our upcoming example. 

Our example will be a so called ``homomorphism problem''. This means that it encodes the problem of finding a homomorphism to a fixed finite graph. We use the same formalization of this as in \cite{G+_trees}:

\begin{definition}\label{def:hom}
    Let $H$ be a finite graph (in the usual sense) with vertex set $V$. $\Pi_H$ is the LCL on $\Delta$-regular graphs $(V,\mathcal{V},\mathcal{E})$, where $\mathcal{V}$ is the set of multisets whose $\Delta$-elements are all the same, and $\mathcal{E}$ is the set of pairs $\{v,w\}$ with $v$ and $w$ adjacent in $H$.
\end{definition}

For example, $\Pi_{K_k}$ is the problem of proper $k$-coloring. We start with a characterization of which homomorphism problems are greedy which may be interesting in its own right. For example, by Theorem \ref{thm:ciffgreedy} and the previous remark about $\Pi_{K_k}$, it generalizes the observation of Schmerl \cite{schmerl_brooks} that every computable maximum degree $\Delta$-graph is computable ($\Delta+1$)-colorable, but that there are such graphs which are not computably $\Delta$-colorable.

\begin{lemma}
    Let $H$ be a finite graph in the usual sense. $\Pi_H$ is greedy if and only if $H$ contains a $(\Delta+1)$-clique.
\end{lemma}

\begin{proof}
    If $\Sigma'$ is a size $\Delta+1$ set of vertices on which $H$ induces a clique, then this set clearly witnesses that $\Pi_H$ is greedy.
    
    On the other hand, suppose $\Sigma'$ is a set of vertices witnessing that $\Pi_H$ is greedy. We will produce inductively a sequence $v_0,\ldots,v_{\Delta}$ on which $H$ induces a clique, maintaining inductively that $v_i \in \Sigma'$ for $i < \Delta$. 
    
    We use the notation of Definiton \ref{def:greedy}. Let $k \leq \Delta$ and suppose we have $v_i \in \Sigma'$ for $i < k$, all pairwise adjacent. Apply the definition with this value of $k$ and $c(y_i,x) = v_i$ for all $i$. We get an extension to the half edges incident to $x$ in which, by definition of $\Pi_H$, all these half edges must get the same label, call it $v_k$, and $v_k$ must be adjacent to each previous $v_i$. Furthermore, if $k < \Delta$, $x$ has some incident virtual edge, and so $v_k$ must be in $\Sigma'$, maintaining our inductive hypothesis.
\end{proof}

It now suffices to produce a finite graph $H$ with no $(\Delta+1)$-clique but with $\Pi_H \in \textnormal{CONTINUOUS}$. We will do this by induction on $\Delta$, and thus call the graph $H_\Delta$.

The base case $H_2$ will be a 5 cycle, say with vertices $v_0,\ldots,v_4$, with $v_i$ and $v_j$ adjacent if and only if $i$ and $j$ differ by one mod 5. 

Given $H_\Delta$, we define $H_{\Delta+1}$ by introducing a new vertex, call it $w_{\Delta+1}$, which is adjacent to each vertex of $H_\Delta$.

\begin{claim}
For each $\Delta$, $H_\Delta$ does not contain a $(\Delta+1)$-clique.
\end{claim}

\begin{proof}
We proceed by induction on $\Delta$. $H_2$, a 5 cycle, does not contain a triangle. Suppose we know that $H_\Delta$ contains no $(\Delta+1)$-clique, and suppose to the contrary that $H_{\Delta+1}$ contains a $(\Delta+2)$-clique $K$. Then since $H_{\Delta+1}$ has only one new vertex compared to $H_\Delta$, the restriction of $K$ to $H_\Delta$ must give either a $(\Delta+1)$ or $(\Delta+2)$ clique depending on whether $w_{\Delta+1} \in K$. In either case this contradicts the inductive hypothesis.
\end{proof}

\begin{claim}
For each $\Delta$, $H_\Delta \in \textnormal{CONTINUOUS}$ (for $\Delta$-regular forests).
\end{claim}

\begin{proof}
For this proof we abandon talk of half edges and switch to the usual notion of a graph. We prove by induction on $\Delta$ that every continuous forest with maximum degree $\leq \Delta$ admits a continuous homomorphism to $H_\Delta$.

For the base case, let $G$ be a continuous forest on a Polish space $X$ with maximum degree 2. By \cite{bernshteyn_continuous}, we can find a clopen maximal 4-discrete set $A$. Let $G'$ be the continuous graph $G \res (X \setminus A)$. Also by \cite{bernshteyn_continuous}, let $d:(X \setminus A) \rightarrow \omega$ by a continuous coloring such that each $d\inv(n)$ is 7-discrete. Note that each $G'$-component is a path of length at most 7, and so $d$ labels the vertices of each component with unique labels.

We now define our continuous homomorphism, call it $c$. Set $c(x) = v_0$ for $x \in A$. Now consider a $G'$-component, call it $P$. $P$ is a path of length at most 7. The endpoints are possibly $G$-adjacent to points in $A$, and if both are so, the length of $P$ is at least 3. Observe then that we can always extend $c$ to $P$ while keeping it a homomorphism to $H_2$: This is trivial if we are not in the case where both endpoints of $P$ are adjacent to a point in $A$, and if we are in that case, we can alternate between $v_0$ and $v_1$ along $P$ if the length of $P$ is even, or circle the 5-cycle $H_2$ once then do this alternation if the length is odd. (Note that $P$ is long enough to allow this.) Furthermore, since $d \res P$ is injective, we can do this in a constructive way using $d$ to break symmetry. That is, we can extend $c$ to a homomorphism $G \res H_2$ in such a way that the value of $c$ at a point $x$ depends only on the values of $d$ and the indicator funciton $\chi_A$ of $A$ on $B(x,N)$ for some large constant $N$. It follows from this and the continuity of $d$ and $\chi_A$ that $c$ is continuous. (See, for example, Theorem 3.2 in \cite{BCG_Trees2}.)

The inductive step is now easy. Suppose we have this result for some $\Delta \geq 2$, and let $G$ be a continuous forest on a Polish space $X$ with maximum degree $\leq \Delta+1$. We will define a continuous homomorphism $c$ to $H_{\Delta+1}$. By \cite{bernshteyn_continuous} again, we may define $c$ be $w_{\Delta+1}$ on a clopen maximal independent set. Let $G' = G \res (X \setminus c\inv(w_{\Delta+1}))$. This is a continuous graph with maximum degree $\leq \Delta$ by maximality of $c\inv(w_{\Delta+1})$. Thus by inductive hypothesis there is a continuous homomorphism $c':G' \res H_\Delta$. Since $w_{\Delta+1}$ is adjacent to every vertex in $H_\Delta$, we can just extend $c$ to all of $X$ by setting $c \res (X \setminus A) = c'$.
\end{proof}

\section*{Acknowledgements} 

Thanks to the American Institute of Mathematics for hosting the workshop at which much of this research was conducted, as well as to the organizers of that workshop, Clinton Conley, Stephen Jackson, Andrew Marks, and Slawomir Solecki. Thanks also to Jan Greb\'{i}k for suggesting that Theorem \ref{th:main} ought to be true, and encouraging me to write up the proof of it.

\bibliographystyle{amsalpha} 

\bibliography{main}

\end{document}